\documentclass[11pt]{amsart}

\usepackage[utf8]{inputenc}
\usepackage{amssymb, amsthm, amsfonts, amsmath, mathtools}
\RequirePackage[dvipsnames,usenames]{xcolor}
\usepackage{hyperref}
\usepackage{tikz}
\usetikzlibrary{cd,matrix,arrows,positioning} 
\usepackage{comment}
\hypersetup{
	colorlinks,backref,hyperindex,
	linkcolor=Sepia,
	anchorcolor=BurntOrange,
	citecolor=MidnightBlue,
	citecolor=OliveGreen,
	filecolor=BlueViolet,
	menucolor=Yellow,
	urlcolor=OliveGreen
}

\usepackage{kantlipsum} 

\calclayout

\DeclareMathOperator{\Spec}{Spec}
\DeclareMathOperator{\Proj}{Proj}

\newcommand\ddfrac[2]{\frac{\displaystyle #1}{\displaystyle #2}}

\newcommand{\Supp}{\textup{Supp}}

\newcommand{\charac}{\textup{char}}
\newcommand{\pr}{\textup{pr}}

\newcommand{\Fr}{\textup{Fr}}

\newcommand{\id}{\textup{id}}
\newcommand{\Pic}{\textup{Pic}}

\newcommand\reallywidehat[1]{%
\savestack{\tmpbox}{\stretchto{%
  \scaleto{%
    \scalerel*[\widthof{\ensuremath{#1}}]{\kern-.6pt\bigwedge\kern-.6pt}%
    {\rule[-\textheight/2]{1ex}{\textheight}}
  }{\textheight}%
}{0.5ex}}%
\stackon[1pt]{#1}{\tmpbox}%
}
\newcommand{\bA}{\mathbb{A}}

\newcommand{\cO}{\mathcal{O}}
\newcommand{\bP}{\mathbb{P}}
\newcommand{\bQ}{\mathbb{Q}}

\newtheorem{theorem}{Theorem}

\newtheorem{lemma}[theorem]{Lemma}
\theoremstyle{definition}

\newtheorem{remark}[theorem]{Remark}

\title{Invariance of plurigenera fails in positive and mixed characteristic}
\author{Iacopo Brivio}
\address{ National Center for Theoretical Sciences, Taipei, 106, Taiwan}
\email{ibrivio@ncts.ntu.edu.tw}
\date{}

\begin{document}

\begin{abstract}
We construct smooth families of elliptic surface pairs with terminal singularities over a DVR of positive or mixed characteristic $(X,B)\to \Spec R$, such that $P_m(X_k,B_k)>P_m(X_K,B_K)$ for all sufficiently divisible $m>0$. In particular, this shows that invariance of all sufficiently divisible plurigenera does not follow from the MMP and Abundance Conjectures.
\end{abstract}

\maketitle


\section{Introduction}

A famous theorem by Siu (\cite{Siu, Siu2}) states that, if $X\to U$ is a smooth one-parameter family of complex projective varieties, the plurigenera of the fibers $P_m(X_u): =h^0(X_u,mK_{X_u})$ are independent of $u\in U$, for all $m\geq 0$. This result, and its generalizations to the log category (\cite{BP}), play an important role in the construction of moduli spaces for varieties of general type (\cite{HMX}). It is worth noticing that Siu's argument is analytic and to this day there is no algebraic proof of \cite{Siu2}. Notable exceptions are the cases of families of threefolds (\cite{KMFlips}) and of varieties of general type (\cite{Kaw1}).  However, it has been known for a while that invariance of plurigenera follows from the Minimal Model and Abundance Conjectures (\cite{Nak}), at least over the complex numbers.

Much less is known when $U$ is replaced by $\Spec R$, where $R$ is a DVR of positive or mixed characteristic, even for families of surfaces. By \cite[Theorem 9.1]{KU} invariance of Kodaira dimension holds. However, invariance of \textit{all} plurigenera can fail: in \cite{La, KU} the authors constructed families of Enriques, resp. elliptic surfaces $X\to\Spec R$ such that $P_1(X_k)>P_1(X_K)$. In \cite{Suh} it is shown that the same can happen even when $K_X$ is ample. In all these examples however, the equality $P_m(X_k)=P_m(X_K)$ holds for all sufficiently divisible $m\geq 0$. Even better, by \cite{EH} we have that if $(X,B)\to \Spec R$ is a log smooth family of klt surface pairs, then $P_m(X_k,B_k)=P_m(X_K,B_K)$ holds for all sufficiently divisible $m\geq 0$ except, \textit{possibly}, when $\kappa(K_X+B/R)=1$ and $B_k$ is vertical with respect to the Iitaka fibration of $K_{X_k}+B_k$. It is then natural to ask whether the above equality always holds for all such $m$. We show this is not the case.

\begin{theorem}\label{thm: main result}
For every prime $p$, there exist smooth families of minimal elliptic surface pairs of Kodaira dimension one with terminal singularities $(X,B)\to\Spec R$, where $R$ is an excellent DVR, with algebraically closed residue field $k$ of characteristic $p>0$ and fraction field $K$, such that 
\[
P_m(X_k,B_k)>P_m(X_K,B_K)
\]
for all sufficiently divisible $m>0$. Furthermore, for every $p$ and every sufficiently divisible $m>0$, the difference $P_m(X_k,B_k)-P_m(X_K,B_K)$ can be arbitrarily large.
\end{theorem}

Note that, over the complex numbers, invariance of all sufficiently divisible plurigenera holds for smooth families of terminal pairs of non-negative Kodaira dimension (\cite{Paun}). When the family has relative dimension two, this can be quickly shown using the MMP and Abundance Conjecture.

\textbf{Acknowledgments}:  I would like to thank my PhD advisor, Prof. James M\textsuperscript{c}Kernan, for his support and guidance that made this work possible. I have also benefitted from discussions with Prof. Michael M\textsuperscript{c}Quillan and Prof. Jungkai A. Chen, to whom I extend my gratitude. I would also like to thank Prof. Christopher Hacon for reading a preliminary version of this paper, and Fabio Bernasconi for many useful discussions. Lastly, I would like to thank Prof. J\'anos Koll\'ar for carefully reading this paper and for his precious feedback. The author has been supported by NSF research grants no: 1265263 and no: 1802460, and by a grant from the Simons Foundation \#409187.

\textbf{Note from the author}: the previous version of this preprint claimed that Theorem \ref{thm: main result} held even in the absence of a boundary divisor. I have been informed by Prof. J\'anos Koll\'ar that the last part of my construction contains a mistake. We will address the boundary-free case in a future (version of this) paper.

\section{Preliminaries}

\subsection{Notation and conventions}

\begin{itemize}
\item All schemes we consider will be of finite type over their bases.
\item $R$ denotes an excellent DVR with algebraically closed residue field $k$ of characteristic $p>0$, fraction field $K$, and uniformizer $\varpi$.
\item Let $F$ be a field: a \textit{variety} is an integral and separated $F$-scheme. We usually assume our varieties to be normal. A \textit{family of varieties} is an integral $R$-scheme $X\to \Spec R$ whose fibers are varieties. The family is said to be smooth, resp. projective, if $X$ is $R$-smooth, resp. $R$-projective.
\item A \textit{pair} $(X,B)$ consists of a normal integral scheme $X$ with an effective $\bQ$-divisor $B$ such that $K_X+B$ is $\bQ$-Cartier. We refer to \cite{KolSMMP} for the various definitions of singularities of pairs. A \textit{family of pairs} is a pair $(X,B)\to \Spec R$ such that $(X_u,B_u)$ is a pair for all $u\in\Spec R$. The family is said to be smooth, resp. projective, if $X$ is $R$-smooth, resp. $R$-projective.
\item If $(X,B)$ is a projective pair over a field $F$, we denote by $P_m(X,B): =h^0(X,\lfloor m(K_X+B)\rfloor)$ the \textit{m-plurigenus} of $(X,B)$.
\item We consider $\bP^1_R$ with homogeneous coordinates $[S:T]$. We denote by $\bA^1_{R,s},\bA^1_{R,t}\subset\bP^1_R$ the affine open sets $\lbrace T\neq 0\rbrace$ and $\lbrace S\neq 0\rbrace$, respectively. The distinguished $R$-points $\lbrace T=0\rbrace$ and $\lbrace S=0\rbrace$ are denoted by $0$ and $\infty$, respectively.
\item Let $X$ be an $R$-scheme: we denote by $X_k$, $X_K$ and $X_{\overline{K}}$, the special, the generic, and the geometric generic fiber of $X\to \Spec R$, respectively. An analogous notation will be used for sheaves over $X$ and their sections.
\end{itemize}

\subsection{Elliptic surfaces}

Let $F$ be a field: a morphism of smooth projective $F$-varieties $f\colon X\to C$ is an \textit{elliptic surface} if $\dim(X)=2$, $\dim(C)=1$, $f_\ast\cO_X=\cO_C$, and a general fiber is a smooth curve of genus 1. The elliptic surface is said to be \textit{minimal} if the fibers of $f$ do not contain any $(-1)$-curve. Let $f^\ast(c_i)=m_iD_i$ be the multiple singular fibers, where $m_i$ is the gcd of the coefficients of the components of $f^\ast(c_i)$. Since $C$ is a smooth curve, we have a decomposition
\[
R^1f_\ast\cO_X=L\oplus T
\]
where $L$ is a line bundle and $T$ is torsion. We will denote by $t$ the length of $T$. The fibers over $\Supp (T)$ are called \textit{wild} fibers; all of them are multiple (\cite[Proposition 3]{BM}). A multiple fiber which is not wild is called \textit{tame}.

\begin{theorem}[{\cite{BM}}]\label{CBF}
Let $f\colon X\to C$ be a minimal elliptic surface. Then
\[
K_X\sim_{\bQ}f^\ast\left(K_C-L+\sum_i\frac{a_i}{m_i}c_i\right),
\]
where $\deg(-L)=\chi(X,\cO_X)+t$ and $0\leq a_i\leq m_i-1$, with $a_i=m_i-1$ if and only if $f^\ast(c_i)$ is tame. Moreover, $T$ is supported precisely at those points $c\in C$ such that $h^1(f^{-1}(c),\cO_{f^{-1}(c)})>1$.
\end{theorem}

\subsection{Iitaka fibration and invariance of plurigenera} Let $X$ be a normal, integral, and projective $R$-scheme, let $D$ be an effective $\bQ$-Cartier divisor, and consider the rational maps of $R$-schemes
\[
\phi_{|mD|}\colon X\dashrightarrow Z_m\subset \bP H^0(X,mD)^*,
\]
where $Z_m$ denotes the image of $\phi_{|mD|}$. By \cite[Sections 2.1.A, 2.1.B]{La1}, for all $m>0$ sufficiently divisible the maps $\phi_{|mD|}$ are birational to a fixed morphism $\phi_\infty\colon X_\infty\to Z_\infty/R$, called the \textit{Iitaka fibration of $D$}, satisfying $\phi_{\infty,\ast}\cO_{X_\infty}=\cO_{Z_\infty}$. The \textit{Iitaka dimension of $D$ over $R$} is defined to be $\kappa(D/R): =\dim_R Z_\infty$; note that we always have 
\[
\kappa(D/R)=\kappa(D_K)\hspace{10mm}\kappa(D/R)\leq\kappa(D_k)
\]
\noindent
by upper-semicontinuity of cohomology. If $(X,B)$ is a pair over $R$, the \textit{Kodaira dimension of $(X,B)$} is defined to be $\kappa(X,B/R): =\kappa(K_X+B/R)$. We will usually assume $D$ to be \textit{semiample}, i.e. $mD$ is basepoint-free for some $m>0$. The section ring $R(D): =\bigoplus_{m\geq 0}H^0(X,mD)$ is then a finitely generated $R$-algebra and, for all sufficiently divisible $m>0$,
\[
\phi_m=\phi_\infty\colon X\to Z: =\Proj R(D)
\]
\noindent
is the Iitaka fibration of $D$. Note that, if $D$ is semiample, $\kappa(D/R)=\kappa(D_k)$.

The following Lemma is the key to the construction of examples violating invariance of $P_m(X_u,B_u)$ for all sufficiently divisible $m>0$.

\begin{lemma}\label{l-SFiffINV}
Let $X\to\Spec R$ be a projective family of normal varieties, let $D$ be a semiample $\bQ$-Cartier $\bQ$-divisor on $X$, and let $f\colon X\to Z$ be its Iitaka fibration. Then $f_{k,*}\cO_{X_k}=\cO_{Z_k}$ if and only if $h^0(X_u,mD_u)$ is independent of $u\in\Spec R$ for all $m\geq 0$ sufficiently divisible.
\end{lemma}

\begin{proof}
If $h^0(X_u,mD_u)$ is independent of $u\in\Spec R$ for all $m\geq 0$ sufficiently divisible, then for all such $m$ we have surjectivity of the restriction map
\[
H^0(X,mD)\otimes k\to H^0(X_k,mD_k),
\]
thus $f_k$ is the Iitaka fibration of $D_k$. For the reverse implication, write $D\sim_{\bQ}f^*A$ for some ample $\bQ$-divisor on $Z$ and observe that for all $m> 0$ sufficiently divisible and all $u\in\Spec R$ we have $h^0(X_u,mD_u)=\chi (Z_u,mA_u)$ by the projection formula and Serre vanishing. As $Z$ is integral, it is flat over $R$, hence $\chi (Z_u,mA_u)$ is independent of $u\in \Spec R$.
\end{proof}

\begin{remark}\label{r-reducedfiber}
Let $F$ be a field, let $f\colon X\to C$ be a morphism with connected fibers between smooth projective $F$-varieties, and suppose $C$ is a curve. Then we have a Stein factorization
\[
f\colon X\xrightarrow{\overline{f}}\overline{C}\xrightarrow{h}C
\]
where $h$ is a universal homeomorphism. If $\charac (F)=0$ then $h$ is an isomorphism by Zariski's Main Theorem. If $\charac (F)=p>0$ then $h$ is a composition of geometric Frobenius morphisms, hence $f_*\cO_X=\cO_C$ if and only if a general fiber of $f$ is reduced.
\end{remark}

\section{Proof of Theorem \ref{thm: main result}}

We fix an integer $n\geq 1$ and set $q: =p^n$. Let $E\to \Spec R$ be a family of elliptic curves such that $\vert\Pic(E_{\overline{K}})[q]\vert >\vert\Pic(E_k)[q]\vert$. Note that this condition is always satisfied when $R$ is of mixed characteristic, while in equicharacteristic $p$ we may consider an ordinary elliptic curve degenerating to a supersingular one. After possibly replacing $R$ by a finite extension we may then assume that there exists a non-trivial $q$-torsion line bundle $M$ on $E$ such that $M_k=\cO_{E_k}$. Denote by $\mathbf{1}_M\in H^0(E,M^q)$ a nowhere vanishing section.

Consider now the following commutative diagram of $R$-schemes
\begin{center}
\begin{tikzcd}
Y \arrow[d, "\pi",swap] & X \arrow[d, "f"] \arrow[l,"\nu",swap]\\
E\times_R\bP^1_R \arrow[r,"\pr_2"] & \bP_R^1,
\end{tikzcd}
\end{center}
where $\pi$ is the $q$-cyclic cover branched over $\mathbf{1}_M\boxtimes TS^{q-1}\in H^0(E\times_R\bP^1_R,M^q\boxtimes \cO_{\bP^1_R}(q))$, and $\nu$ is the normalization. Over $\bA^1_{R,t}$ we have the following description
\begin{center}
\begin{tikzcd}
\Spec\ddfrac{\cO_E[t,\lambda]}{(\lambda^q-\varphi t)} \arrow[d] & \Spec\ddfrac{\cO_E[t,\lambda]}{(\lambda^q-\varphi t)} \arrow[d] \arrow[l,"=",swap]\\
\Spec\cO_E[t,\lambda] \arrow[r] & \Spec R[t],
\end{tikzcd}
\end{center}
while over $\bA^1_{R,s}$ we have
\begin{center}
\begin{tikzcd}
\Spec\ddfrac{\cO_E[s,\xi]}{(\xi^q-\psi s^{q-1})} \arrow[d] & \Spec\ddfrac{\cO_E[s/\xi,\xi]}{(\xi-\psi (s/\xi)^{q-1})} \arrow[d] \arrow[l]\\
\Spec\cO_E[s,\xi] \arrow[r] & \Spec R[s].
\end{tikzcd}
\end{center}
Here $\varphi,\psi\in\cO^*_E$ are local trivializations of $\mathbf{1}_M$, hence $\varphi_k,\psi_k\in k^*$. It is straightforward to verify that $X$ is smooth over $R$. A general fiber of $f_K$ is the smooth $q$-cover $F_K\to E_K$ induced by $\mathbf{1}_{M_K}\in H^0(E_K,M_K^q)$. By Remark \ref{r-reducedfiber} we then have that $f_{K,*}\cO_{X_K}=\cO_{\bP^1_K}$, hence $f_{*}\cO_{X}=\cO_{\bP^1_R}$ as Stein factorization and flat base-change commute. From the equations we see that the function field extension induced by $f_k$ factors as
\[
k(X_k)\hookleftarrow k(\bP^1_k)^{\frac{1}{q}}\hookleftarrow k(\bP^1_k),
\]
which in turn yields a non-trivial Stein factorization
\[
f_k\colon X_k\xrightarrow{\overline{f_k}}(\bP^1_k)^{(-n)}\xrightarrow{\Fr^n}\bP^1_k.
\]
Observe that $f_K\colon X_K\to\bP^1_K$ is an isotrivial elliptic surface with multiple fibers $qE_K$ over $0_K,\infty_K$. We claim both these fibers are tame: note that we can find isomorphic neighborhoods of $f_K^{-1}(0_K)$ and $f_K^{-1}(\infty_K)$, thus one fiber is tame if and only if the other one is. Note also that $h^1(X_K,\cO_{C_K})=1$. By contradiction, suppose both fibers are wild. By \cite[Theorem III.12.11]{Har} the natural map
\[
R^1f_{K,*}\cO_{X_K}\otimes K(b)\to H^1(X_{K,b},\cO_{X_{K,b}})
\]
is surjective for all $b\in\bP^1_K$ and, as $f_{K,*}\cO_{X_K}=\cO_{\bP^1_K}$, the Leray spectral sequence yields
\[
1=h^1(X_K,\cO_{X_K})=h^0(\bP^1_K,R^1f_{K,*}\cO_{X_K})\geq h^0(\bP^1_K,(R^1f_{K,*}\cO_{X_K})_{\textup{tor}})\geq 4.
\]
As $\chi(X_K,\cO_{X_K})=0$, Theorem \ref{CBF} implies $K_{X_K}\sim_{\bQ}f_K^*(-2+2(q-1)/q)H_K$, where $H$ is an hyperplane on $\bP^1_R$, thus $K_{X}\sim_{\bQ}f^*(-2+2(q-1)/q)H$, as $X_k$ is irreducible.

Let now $z_1,...,z_l$ be pairwise disjoint $R$-points of $\bP_R^1$, not intersecting $0$ or $\infty$, and let $0<\epsilon\ll 1$ be a rational number such that, setting $B: =f^*(\epsilon\sum_i z_i)$, we have that $(X_u,B_u)$ is terminal for all $u\in \Spec R$. For such choice of $\lbrace z_i\rbrace_i$ we have that $\epsilon$ is independent of $l$. Thus upon taking very large $l$ we may also assume that $K_X+B\sim_{\bQ}f^*A$ for some ample $\bQ$-divisor on $\bP^1_R$ of degree $d$. In particular, $f$ is the Iitaka fibration of $K_X+B$, thus $P_m(X_u,B_u)$ will jump for all sufficiently divisible $m>0$, by Lemma \ref{l-SFiffINV}. On the special elliptic surface we have $K_{X_k}+B_k\sim_{\bQ}\overline{f_k}^*(qA)$, hence the projection formula on $f_K$ and $\overline{f_k}$ yields
\[
P_m(X_k,B_k)=h^0(\bP^1_k,\cO_{\bP^1_k}(qmd))>h^0(\bP^1_K,\cO_{\bP^1_K}(md))=P_m(X_K,B_K).
\]
As $q=p^n$ we see that, for every characteristic $p$ and every $m\geq 1$ divisible enough, the jump in plurigenera can be arbitrarily large, thus concluding the proof.

\begin{remark}
By taking products we can construct smooth, higher-dimensional families of terminal pairs $(W,D)\to \Spec R$ with $K_W+D$ semiample and $0< \kappa(W,D/R)<\dim_R W$ such that invariance of all sufficiently divisible plurigenera fails. Consider smooth families of Abelian, resp. canonically polarized varieties, $A\to \Spec R$ and $V\to \Spec R$. Then
\[
g: =f\times\pr_2\colon (W,D): =(X\times_R A\times_R V,B\times_R A\times_R V)\to\bP^1_R\times_R V
\]
is the Iitaka fibration of $K_W+D$. By construction $g_k$ has a non-trivial Stein factorization
\[
g_k\colon W_k\xrightarrow{\overline{f_k}\times\pr_{2,k}}\left(\bP^1_k\right)^{(-n)}\times V_k\xrightarrow{\Fr^n\times\id_{V_k}}\bP^1_k\times V_k,
\]
hence $P_m(W_k,B_k)-P_m(W_K,B_K)>0$ can be arbitrarily large for all sufficiently divisible $m$, as in the surface case.
\end{remark}

\bibliographystyle{alpha}
\bibliography{bib.bib}
\end{document}